\documentclass[a4paper]{article}
\usepackage{amsthm,amsmath,amssymb}
\usepackage{enumerate,enumitem}
\usepackage{graphicx}
\usepackage{xfrac}
\usepackage{tabularx,booktabs,multirow}
\newcolumntype{C}{>{\centering\arraybackslash}X}
\newcolumntype{D}{>{\centering\arraybackslash}X}

\DeclareGraphicsExtensions{.pdf,.png,.eps}
\graphicspath{{figs/}}

\newtheorem{theorem}{Theorem}
\newtheorem{lemma}[theorem]{Lemma}

\newtheorem*{claim*}{Claim}

\theoremstyle{remark}

\newcommand{\B}{\ensuremath{\mathcal{B}}}

\newcommand{\cH}{\ensuremath{\mathcal{H}}}
\newcommand{\cF}{\ensuremath{\mathcal{F}}}

\newcommand{\sat}{{\rm  sat}}

\usepackage[usenames,dvipsnames]{xcolor}

\begin{document}

\title{A note on saturation for Berge-$G$ hypergraphs}

\author{Maria Axenovich\thanks{Karlsruhe Institute of Technology, Karlsruhe, Germany}\and Christian Winter\thanks{Karlsruhe Institute of Technology, Karlsruhe, Germany}\thanks{Research supported in part by
Talenx stipendium.}}

\maketitle

\begin{abstract}
For a graph $G=(V,E)$,  a hypergraph $H$ is called {\it Berge-$G$} if there is a hypergraph $H'$, isomorphic to $H$,  so that $V(G)\subseteq V(H')$  and there is a  bijection $\phi: E(G) \rightarrow E(H')$ such that for each $e\in E(G)$, $e \subseteq \phi(e)$.  The set of all Berge-$G$ hypergraphs is denoted $\B(G)$.

A hypergraph $H$ is called Berge-$G$ {\it saturated} if it does not contain any subhypergraph from $\B(G)$, but adding any new hyperedge of size at least $2$  to $H$ creates such a subhypergraph. 

Since each Berge-$G$ hypergraph contains $|E(G)|$ hypergedges, it follows that each Berge-$G$ saturated hypergraph must have at least $|E(G)|-1$ edges. 
We show that for each graph $G$ that is not a certain star and for any $n\geq |V(G)|$, there are Berge-$G$ saturated hypergraphs on $n$ vertices and exactly $|E(G)|-1$ hyperedges. 
This solves a problem of finding a saturated hypergraph with the  smallest number of edges 
 exactly.

\end{abstract}

\section{Introduction}

For a graph $G=(V,E)$,  a hypergraph $H$ is called {\it Berge-$G$} if there is a  hypergraph $H'$, isomorphic to $H$, so that $V(G)\subseteq V(H')$  and there is a  bijection $\phi: E(G) \rightarrow E(H')$ such that for each $e\in E(G)$, $e \subseteq \phi(e)$.  The set of all Berge-$G$ hypergraphs is denoted $\B(G)$.

 Here, for a graph or a hypergraph $F$, we shall always denote the vertex set of $F$ as $V(F)$ and the edge set of $F$ as $E(F)$. A copy of a graph $F$ in a graph $G$ is a subgraph of $G$ isomorphic to $F$. When clear from context, we shall drop the word ``copy" and just say that there is an $F$ in $G$.

Several classical questions regarding Berge-$G$ hypergraphs have been considered. Among those are extremal numbers for Berge-$G$ hypergraphs measuring the largest number of hyperedges or the largest weight of hypergraphs on $n$ vertices that contain no subhypergraph from $\B(G)$, see for example \cite{GP, G, GMT, PTTW}. 
In addition, Ramsey numbers for Berge-$G$ hypergraphs have been considered in  \cite{AG, GYS, GYLSS}.

In this paper, we consider a saturation problem. 
Let $\cF$ be a class of  hypergraphs with edges of size at least two. A hypergraph $\cH$ is called $\cF$ {\it saturated} if it does not contain any subhypergraph isomorphic to a member of $\cF$, but adding any new hyperedge of size at least $2$ to $\cH$ creates such a subhypergraph.  

Saturation problem for families of  $k$-uniform  hypergraphs has been treated by Pikhurko \cite{P}, see also \cite{P1}.  Pikhurko \cite{P} proved in particular, that for any $k$-uniform hypergraph $G$ there is an $n$-vertex $k$-uniform hypergraph $H$ that is $\{G\}$ saturated  and has $O(n^{k-1})$ edges. This extends a result of K\'aszonyi and Tuza \cite{KT} who proved this fact for $k=2$, i.e., for graphs. 
See also a survey of Faudree et al. \cite{FFS}.  Here $\{G\}$ saturated means that $H$ has no subhypergraph isomorphic to $G$ but adding any new hyperedge of size $k$ creates such a subhypergraph. This result is asymptotically tight for some $G$. The determination of a smallest size for $\{G\}$-saturated hypergraphs remains open in general.  In the same setting of $k$-uniform hypergraphs, English et al. \cite{EGMT} proved that there are $\B_k(G)$ saturated hypergraphs on $n$ vertices and $O(n)$ hyperedges, where $\B_k(G)$ is the set of all $k$-uniform Berge-$G$ hypergraphs, $3\leq k\leq 5$. See also English et al. \cite{EGKMS}, for Berge saturation results on some special graphs.  \\

We restrict our attention to the non-uniform case and Berge-$G$ hypergraphs. 
For $n\geq |V(G)|$,  let the {\it saturation number} for a Berge-$G$ hypergraph be defined as 
$$\sat(n, \B(G))= \min \{|E(\cH)|: ~ \cH \mbox{ is a }  \B(G)  \mbox{ saturated hypergraph on } n  \mbox{ vertices} \}.$$

Observe that for any nontrivial graph $G$, $$\sat(n, \B(G))\geq |E(G)|-1.$$

Since no Berge-$G$ hypergraph has hyperedges of sizes less than $2$, we can assume that all hypergraphs considered have hyperedges of sizes at least $2$.  We further assume that graphs considered have no isolated vertices. The following is the main result of this paper:

\begin{theorem}\label{thm:main}
 Let $G=(V, E)$ be a graph with no isolated vertices,  $n\geq |V(G)|$, and  $m=|E(G)|-1$. Then 
$$
\sat(n, \B(G))=
\begin{cases}
|E(G)|, & \mbox{ if } G \mbox{ is a star on at least four edges},\\
|E(G)| -1, & \mbox{ otherwise.}
\end{cases}$$
Moreover if $G_1$ is a star on at least $4$ edges  and $G_2 $ is  any other graph,  then   $\cH_t(n)$ 
 and  $\cH(n, m)$  are a Berge-$G_1$ and a Berge-$G_2$ saturated hypergraphs, respectively.
\end{theorem}

\noindent
For a positive integer $n$, let $[n]=\{1, 2, \ldots, n\}$.  We shorten $\{i,j\}$ as $ij$ when clear from context. If $F$ is a hypergraph and $e$ is a hyperedge, we denote by $F+e$, $F-e$, a hypergraph obtained by adding $e$ to $F$, deleting $e$ from $F$, respectively.\\

\noindent
{\bf Construction of  a hypergraph  $\cH_t(n)$:}\\
Let $n$ and $t$ be positive integers, $t\leq n$.  Let $ \cH_t(n) = ([n],\{[n], [n]-\{1\}, [n]-\{2\},\ldots, [n]-\{t-3\}, [t-3]\}).$\\

\noindent
{\bf Construction of a set system $H'(n,m)$ and a hypergraph $\cH(n, m)$:}\\
Let $n$ and $m$ be positive integers, $m\leq \binom{n}{2}$.  Let $x= \min\{ m-1,  n\}$. 
Let $V' $ be a set of singletons, $V'\subseteq \{\{i\}\colon i\in[n]\}$, $|V'|=x$.
Let $E'$ be an edge-set of an almost regular graph (the degrees of vertices differ by at most one) on the vertex set $[n]$,  such that $|E'|=m - x-1$. 
Let $H' (n,m)= \{\varnothing\} \cup V' \cup E'$.\\

Informally, we build a set system $H'(n,m)$ of $m$ sets on the ground set $[n]$  by first picking an 
empty set, then as many as possible singletons, and then pairs, so that the pairs form an edge-set of an almost  regular graph.\\

Let $$\cH=\cH(n,m)= ([n], \{ [n]-E:  ~ E\in H'(n,m)\}).$$ 
Note that $|E(\cH)|=m$ and each hyperedge of $\cH$ has size $n$, $n-1$, or $n-2$.\\

\noindent
\textbf {Examples.} \\
If $n=4$ and $m=4$, we have: 
\begin{eqnarray*}
H'(4,4)& = &\{ \varnothing, \{1\}, \{2\}, \{3\}\},\\
E(\cH(4,4))& = &  \{ [4], \{2,3,4\}, \{1,3,4 \}, \{1,2,4\}\}.
\end{eqnarray*}

\noindent
If $n=5$ and $m = 8$, we have
$$H'(5,8)= \{ \varnothing,  \{1\}, \{2\}, \{3\}, \{4\}, \{5\}, \{12\}, \{34\}\},$$
$$E(\cH(5,8)) =  \{ [5], \{2,3,4,5\}, \{1,3,4,5 \}, \{1,2, 4,5\}, \{1,2,3,5\}, \{1,2, 3,4\}, \{3,4, 5\}, \{1,2,5\}\}.$$

Let $H$ be a Berge-$G$ hypergraph, we call  a copy $G'$ of $G$, where $V(G')\subseteq V(H)$  and the edges of $G'$ are contained in distinct hyperedges of $H$,  an {\it underlying  graph} of the Berge-$G$ hypergraph $H$. For example,  if $G'$ 
 is a triangle on vertices $1,2,3$, then a hypergraph $(\{1,2, 3, 4\}, \{\{1,2\}, \{2, 3, 4\}, \{1,2, 3, 4\}\})$ is Berge-$K_3$ and $G'$ 
 is an underlying graph of Berge-$K_3$ hypergraph $H$.

\section{Proof of the main theorem}

Let $S_t$ denote a star on $t$ vertices. 

\begin{lemma}
Let $t\ge 5$, $n\geq t$. Then $\sat(n,\B(S_t))=t-1= |E(S_t)|$.
\label{lem:35}
\end{lemma}
\begin{proof}

To show the lower bound, assume first that there is a hypergraph  $\cH$  on $t-2$ hyperedges and vertex set $[n]$ that is Berge-$S_t$ saturated. 
Since maximum degree of any member in $\B(S_t)$ is at least $t-1$, we have that the maximum degree of $\cH+e$ for any new edge $e$ of size at least $2$ is at least $t-1$.
We have that   $\cH$ has at least $|V(S_t)|=t\geq 5$ vertices.
Assume first that $\cH$ contains an edge of size $2$, say $12$. Then any vertex in $\{3, \ldots, n\}$ does not belong to this edge, so it has a degree at most $t-3$.
Thus, for any $i,j\in \{3, \ldots, n\}$, $i\neq j$, the maximum degree of $\cH+ij$ is at most $t-2$, implying that $ij \in E(\cH)$.
Since the edge $12$ was chosen arbitrarily, we can conclude that $\cH$ contains all edges of size $2$. Thus $\cH$ has at least $\binom{n}{2} \geq \binom{t}{2} > t-2$ edges, a contradiction. 
Therefore  $\cH$ has no hyperedges of size $2$.
Assume next that all but at most one vertex, say $n$, belong to all hyperedges of $\cH$. Thus each hyperedge contains the set $[n-1]$, implying that each hyperedge is either $[n-1]$ or $[n]$, a contradiction to the fact that there are  $t-2\geq 3$ distinct hyperedges in $E(\cH)$.
Hence, there are two vertices, say $1$ and $2$,  each with degree at most $t-3$. 
We know that $12 \not\in E(\cH)$ and that $\cH+12$ has maximum degree at most $t-2$, a contradiction. 
Thus $\cH$ is not $\B(S_t)$ saturated. \\

For the upper bound, we show that  $\cH_t$ is a $\B(S_t)$-saturated hypergraph. Recall that  $\cH= \cH_t= ([n],\{[n], [n]-\{1\}, [n]-\{2\},\ldots, [n]-\{t-3\}, [t-3]\}).$

Note that each vertex of $\cH$ has degree $t-2$. Thus $\cH$ is $\B(S_t)$-free.
Let $e\subseteq [n]$ of size at least $2$, such that $e\not\in E(\cH)$. Let $i,j \in e$, $i\neq j$. We shall show that  $\cH+e$ contains a Berge $S_t$ hypergraph.\\

Case 1.  $i, j\in [t-3]$, without loss of generality $i=1, j=2$.
Then the pairs $1n, 1(n-1), 13, \ldots, 1(t-2), 12$ are 
contained in $[n],  [n]-\{2\},\ldots, [n]-\{t-3\}, [t-3],  e$, respectively, and form an underlying graph of Berge-$S_t$ in $\cH+e$.\\

Case 2.  $i$ or $j$ is not in $[t-3]$. Let, without loss of generality $i=n$. Then, without loss of generality $j=n-1$ or $j=1$.
Then the pairs  $n2, n3, \ldots, n(t-2)$ are contained in 
$[n]-\{1\}, [n]-\{2\},\ldots, [n]-\{t-3\}$, respectively, and 
and the pairs $1n, (n-1)n$ are contained in $[n], e$ or $e, [n]$, respectively.
Thus all these $t-1$ pairs form an underlying graph  of Berge-$S_t$ in $\cH+e$.
\end{proof}

\vskip 1cm

\begin{proof}[Proof of Theorem \ref{thm:main}]
First we consider some special graphs: stars on at most three edges and a triangle. For the upper bounds on $\sat(n,\B(G))$ for $G=S_2, S_3, S_4, K_3$, consider 
the following hypergraphs in order for $n\geq 2$, $n\geq 3$, $n\geq 4$, and $n\geq 3$, respectively:
$([n], \emptyset), ([n], \{[n]\}), ([n], \{[n], [n]-\{1\}\}),  ([n], \{[n], [n]-\{1\}\})$. 
It is easy to see that these hypergraphs are saturated for the respective Berge hypergraphs.
Thus, for $G$ being one of these graphs, $\sat(\B(G))\leq |E(G)|-1$. Since the lower bound on $\sat(\B(G))$ is trivially $|E(G)|-1$, the theorem holds in this case. 
Lemma \ref{lem:35} implies that the theorem holds for all other stars.\\

From now on, let $G$ be a  non-empty graph  which is neither a star nor a $K_3$. Let $n$ be the number of vertices in $G$, $n\geq 4$.
We shall further assume that $G$ has no isolated vertices and that $V(G)=[n]$.
Let $m=|E(G)|-1$. We shall prove that $\cH=\cH(n,m)$ as defined in the introduction is a Berge-$G$ saturated hypergraph, i.e. such that it does not contain any member of $\B(G)$ as a subhypergraph and 
such that for any new hyperedge $e$ of size at least two,  $\cH+e$ contains a Berge-$G$ sub-hypergraph. In fact, instead of $\cH(n,m)$ we shall be mostly using the system $H'(n,m)$ also defined in the introduction.
Note that $\cH$ does not contain any member of $\B(G)$ since $\cH$ has $|E(G)|-1$ edges.
\\

Consider $e$,  $e\subseteq [n]$,  $|e|\geq 2$, $e\not\in E(H)$. Let $\{i,j\}\subseteq e$, $i\neq j$. Relabel the vertices of $G$ such that $ij \in E(G)$ and  $i$ is a vertex of maximum degree in $G$. We shall show that $\cH$ is a Berge-$(G-ij)$, thus showing that $\cH+e$ is Berge-$G$. 
We shall prove one of the following equivalent statements: \\

\noindent
(i)  there is a  bijection $\phi$ between $E(G-ij)$ and $E(\cH)$ such that $e' \subseteq \phi(e')$ for any $e'\in E(G-ij)$,\\
(ii)  there is a bijection $f$ between $E(G-ij)$ and $H'=H'(n,m)$ such that 
for each $e'\in E(G-ij)$, $e' \cap  f(e') = \varnothing$,\\
(iii) there is a perfect matching in a  bipartite graph $F$ with one part  $A= E(G) - \{ij\}$ and the other part $B=H'$ such that 
$e'\in  A=E(G) - \{ij\}$ and $e''\in B= H'$ are adjacent in $F$ iff $e'\cap e'' = \varnothing$. \\

One can see that  (i) and  (ii) are equivalent by defining  $\phi(e') $ to be 
$[n]-f(e')$.  The  equivalence of (ii) and (iii) is  clear   since $|A|=|B|$.
Next, we shall prove  (iii).\\

In each of the cases below, we assume that there is no perfect matching in $F$, thus by Hall's theorem, there is  a set $S\subseteq A$ such that $|N_F(S)|<|S|$.
Let $Q = B \setminus N_F(S)$. We see that each element of $Q$ intersects each edge in  $ S$. Let $G_S$ be a subgraph of $G$ with edge set $S$.
Since each element in $Q$ has size one or two, $G_S$ has a vertex cover of size one or two.  Thus $G_S$ is either a star, a triangle, or
an edge-disjoint union of two stars.
Clearly, $\emptyset$ is not in $Q$. Assume some singleton, say $\{1\}$ is in $Q$. Then $S$ forms a star with center $1$.
Then all singletons $\{2\}, \{3\}, ...$ and $\emptyset$ are in $N_F(S)$. 
If $S\neq A$, i.e., $|E(G)|-1>|S|$, then $|N_F(S)|\geq |S|$, a contradiction to our assumption on $S$.
If $S=A$, i.e., $G$ is a union of a star and an edge $ij$, since $i$ is a vertex of maximum degree in $G$, we see that $G$ is a star, a contradiction.
Thus we can assume that $Q$ contains only two-elements sets, i.e., in particular $H'$ has two-element sets and thus, by definition of $H'$,  $|H'| >n+1$.
Finally, since an empty set and all singletons are not in $Q$, they are in $N_F(S)$, so $|N_F(S)|\geq n$. 
Thus $|S|\geq n+1$, and in particular, $S$ does not form a star. We observed earlier that we could assume that $G$ is not a star.
\\\\

{\bf Case 1.} $G$  is a union of two stars.

We already excluded the case that $G$ is a star, so let's assume that $G$ is an edge-disjoint union of two stars with different centers.
If one of the stars has at most two edges, then $|E(G)| \leq n+1$, and $|S|\leq n$, a contradiction.
Thus  each of the stars has  at least $3$ edges.

Note that $G$ has at most $2n-1$ edges. In particular,  since there are $n$ singletons and an empty set in $H'$ and $|H'|\leq2n-1$,  we have that $E'$, the set of pairs from $H'$,  has size at most $n-2$  and thus the graph on edge set $E'$ has maximum degree at most $2$.  This implies that for every vertex there is a non-adjacent vertex in a graph with edge-set $E'$. Let $k$ be the integer such that $ik\not\in E'$. Relabel the vertices of $G$ such that $ij$ is an edge of $G$, and $i$ and $k$ are the centers of the stars whose union is $G$, and $j\neq k$.
Since $ik\not\in E'$, it follows that $ik \not\in Q$.  Since each pair from $Q$ forms a vertex cover of $G_S$, there is a pair different from $ik$ that forms a vertex cover of $G_S$. 
Since $ik$ is a vertex cover of $G$, it is a vertex cover of $G_S$. Thus $G_S$ has two distinct vertex covers of size $2$.
Then $G_S$ is a subgraph of a triangle with possibly  some further edges incident to the same vertex of the triangle or a subgraph of a $C_4$.
This implies that $|S|\leq n$, a contradiction.\\

\textbf {Case 2.}  $G$ is not a union of two stars.

If $|Q|=1$, then $|N_F(S)|=|B|-1= |A|-1$. Since $|S|> |N_F(S)|= |A|-1$ and $S\subseteq A$, we have that $S=A$, hence $G_S= G - ij$. 
Since there is a vertex cover of $G_S$ of size $2$, we have that $G_S= G-ij$  is a union of two stars $S', S''$, so $G$ is a union of two stars and an edge incident to a vertex of maximum degree of $G$. If maximum degree of $G$ is at least four, then $i$ is a center of $S'$ and $S''$.  Thus $G$ is a union of two stars, a contradiction.
If the maximum degree of $G$  is at most $3$, then $|E(G)| \leq  7$. 
On the other hand, $m=|H'| \ge n+2$.  Thus $n+3\le |E(G)|\le 7$. 
Thus $n=|V(G)| \leq 4$
and for each such choice of $n$ we reach a contradiction by the fact that $n+3\leq |E(G)|$.
If $Q$ contains two disjoint edges, say $12$ and $34$, then $G_S$ can only be a subgraph of a $4$-cycle $13241$. 
So, $|S|\leq 4\leq n$, a contradiction to our assumption that $|S|\geq n+1$. 
\\

Thus $Q$ contains edges that either form a star on at least three edges or a subgraph of a triangle.
If the edges of $Q$ form a star on at least three edges, say $12, 13, 14, \ldots$, $S$ forms a star with center $1$, a contradiction.
If the edges of $Q$ form a triangle, say $123$, then we arrive at a contradiction since no two-element set  can at the same time intersect $12$, $23$, and $13$.
Thus $Q$ contains exactly two adjacent edges, say $12$ and $13$. It follows that $S$ forms a star with center $1$ and maybe an edge $23$. Then $|S|\leq n$, a contradiction.
Hence, there is a perfect matching in $F$ and thus $\cH$ is Berge-$G$ saturated. 
\end{proof}

\section{Conclusions}
In this note, we completely determine $\sat(n, \B(G))$ for any $n\geq |V(G)|$ and show in particular that this function does not depend on $n$.
There are many variations of saturation numbers for non-uniform hypergraphs that could be considered. Among those are functions optimising the total weight of a saturated hypergraph, i.e., the sum of cardinalities of all hyperedges, or functions optimising the size of a saturated multihypergraph. These have been considered by the second author in \cite{W}.
One particularly interesting variation considered in \cite{W} is the following notion of saturation:  a hypergraph $\cH$ is called strongly $\cF$ saturated with respect to a family of hypergraphs $\cF$ if $\cH$ does not contain any member of $\cF$ as a subhypergraph, but replacing any hyperedge $e$ of $\cH$ with $e\cup \{v\}$ for any vertex $v\not\in e$ such that $e\cup \{v\} \notin E(\cH)$ 
creates such a member of $\cF$.\\

\noindent
{\bf Acknowledgements~~~}  We thank Casey Tompkins for useful discussions and carefully reading the manuscript.

\end{document}